\documentclass[12pt,reqno]{amsart}

\usepackage[cp1251]{inputenc}
\usepackage{amsthm,amssymb,amsmath,amsfonts}
\usepackage{graphicx}
\usepackage[matrix,arrow,curve]{xy}
\usepackage{longtable}
\newcommand{\Q}{\ensuremath{\mathbb{Q}}}

\newcommand{\Z}{\ensuremath{\mathbb{Z}}}

\newcommand{\F}{\ensuremath{\mathbb{F}}}
\newcommand{\Fr}{\ensuremath{\mathbf{F}}}

\newcommand{\ka}{\ensuremath{\Bbbk}}

\newcommand{\kka}{\ensuremath{\overline{\Bbbk}}}

\newcommand{\XX}{{\ensuremath{\overline{X}}}}

\newcommand{\Pro}{\ensuremath{\mathbb{P}}}

\newcommand{\Aut}{\ensuremath{\operatorname{Aut}}}

\newcommand{\Gal}{\ensuremath{\operatorname{Gal}}}

\newcommand{\Pic}{\ensuremath{\operatorname{Pic}}}

\newcommand{\ord}{\ensuremath{\operatorname{ord}}}

\makeatletter
\@addtoreset{equation}{section}
\makeatother

\newtheorem{theorem}[equation]{Theorem}
\newtheorem{proposition}[equation]{Proposition}
\newtheorem{lemma}[equation]{Lemma}
\newtheorem{corollary}[equation]{Corollary}

\theoremstyle{definition}
\newtheorem{example}[equation]{Example}
\newtheorem{definition}[equation]{Definition}

\theoremstyle{remark}
\newtheorem{remark}[equation]{Remark}

\title{Minimal del Pezzo surfaces of degree $2$ over finite fields}

\thanks{The research was carried out at the IITP RAS at the expense of the Russian Foundation for Sciences (project $N^{\underline{o}}$ 14-50-00150).}

\author{Andrey Trepalin}

\address{\emph{Andrey Trepalin}
\newline
\textnormal{Institute for Information Transmission Problems, 19 Bolshoy Karetnyi side-str., Moscow 127994, Russia}
\newline
\textnormal{\texttt{trepalin@mccme.ru}}}

\begin{document}

\begin{abstract}
Let $X$ be a minimal del Pezzo surface of degree $2$ over a finite field $\F_q$. The image $\Gamma$ of the Galois group $\Gal(\overline{\F}_q / \F_q)$ in the group $\Aut(\Pic(\XX))$ is a cyclic subgroup of the Weyl group $W(E_7)$. There are $60$ conjugacy classes of cyclic subgroups in $W(E_7)$ and $18$ of them correspond to minimal del Pezzo surfaces. In this paper we study which possibilities of these subgroups for minimal del Pezzo surfaces of degree $2$ can be achieved for given $q$.
\end{abstract}

\maketitle
\section{Introduction}

Let $X$ be a del Pezzo surface of degree $d$ over a finite field $\F_q$, and $\XX = X \otimes \overline{\F}_q$. The image $\Gamma$ of the Galois group $\Gal(\overline{\F}_q / \F_q)$ in the group $\Aut(\Pic(\XX))$ is a cyclic group, which preserves the intersection form. There are finitely many conjugacy classes of cyclic subgroups in the subgroup $\Aut(\Pic(\XX))$ preserving the intersection form. The natural question is which of these classes can realise the group $\Gamma$ for given $q$.

A surface $S$ is called \textit{minimal} if any birational morphism $S \rightarrow S'$ is an isomorphism. The minimality of $X$ can be described in terms of $\Gamma$-action on $\Pic(\XX)$. If $X$ is not a minimal surface then it is isomorphic to a blowup of surface $Y$ at number of points. In this case the action of the group $\Gal(\overline{\F}_q / \F_q)$ on $\Pic(\XX)$ is prescribed by the action of $\Gal(\overline{\F}_q / \F_q)$ on $\Pic(\overline{Y})$ and the degrees of the points of blowup. Therefore the cases of $\Gamma$ for which $X$ is minimal are most interesting for us.

If $X$ is a minimal geometrically rational surface then either $X$ admits a conic bundle structure or $X$ is a del Pezzo surface with the Picard number $\rho(X) = \operatorname{rk} \Pic(X) = 1$ (see \cite[Theorem 1]{Isk79}). In the paper \cite{Ry05} it is shown how the group $\Gamma$ can act on the components of singular fibres of a minimal conic bundle, and for all possibilities of $\Gamma$ corresponding minimal conic bundles are constructed. Del Pezzo surfaces of degree greater than~$4$ are $\F_q$-rational (see {\cite[Chapter 4]{Isk96}}). Therefore minimal del Pezzo surfaces of degree greater than $4$ can be constructed by blowing up some points on $\Pro^2_{\F_q}$ and contracting some exceptional curves. All types of minimal del Pezzo surfaces of degree $4$ are constructed in \cite[Theorem~3.2]{Ry05}. One case of minimal cubic surfaces is constructed in \cite{SD10} for any $q$. The other cases of minimal cubic surfaces are constructed in the paper \cite{RT16} but there are some restrictions on $q$. In the paper \cite{BFL16} for del Pezzo surfaces of $3$, $2$ and $1$ and any~$q$ it is shown how many $\F_q$-points can a surface have. Some results of the paper \cite{BFL16} give constructions of minimal surfaces for certain $\Gamma$. Also it is shown that for any $\Gamma$ there exists the corresponding surface for any sufficiently big $q$ (see \cite[Theorem~1.7]{BFL16}). 

The aim of this paper is to construct minimal del Pezzo surfaces of degree $2$ with given cyclic group $\Gamma$. Note that in this case the group of automorphisms of $\Pic(\XX)$, preserving the intersection form, is the Weyl group $W(E_7)$. The conjugacy classes of elements in this group are well-known (see \cite{Car72}). For convenience of the reader we give a table of these conjugacy classes and some of their properties in Appendix A. We have $18$ conjugacy classes in $W(E_7)$ for which $X$ is minimal. For $6$ of those classes the invariant Picard number $\rho(\XX)^{\Gamma} = 2$ and $X$ admits a conic bundle structure. For the other $12$ conjugacy classes $\rho(\XX)^{\Gamma} = 1$ and $X$ does not admit a structure of a conic bundle.

The considered problem is closely related to zeta-functions. Let $N_d$ be the order of the set $X(\F_{q^d})$. The zeta-function of $X$ is the formal power series
$$
Z_X(t)=\exp\left(\sum_{d=1}^\infty \frac{N_dt^d}{d}\right).
$$
For a rational surface $X$ one has (see~\cite[IV.5]{Man74})
$$
Z_X(t)=\frac{1}{(1-t)P(t)(1-q^2t)}
$$
\noindent where
$$
P(t)=\det(1-qt\Fr|\Pic(\XX)\otimes\Q),
$$
\noindent and $\Fr$ is a linear automorphism of $\Pic(\XX)\otimes\Q$ induced by the Frobenius element. Therefore the zeta-function of a surface $X$ is totally defined by the group $\Gamma$. Moreover, for each cyclic subgroup of $W(E_7)$ we can write down such function. But it is not known whether a given zeta-function corresponding to a subgroup of $W(E_7)$ can be realised by a del Pezzo surface of degree $2$.

This paper gives an answer for this question for minimal del Pezzo surfaces of degree~$2$. In the notation of Table \ref{table1} these surfaces have types $31$, $35$, $40$, $43$--$45$ and $49$--$60$. The main result of this paper is the following.

\begin{theorem}
\label{main}
In the notation of Table \ref{table1} the following holds.

\begin{enumerate}
\item A del Pezzo surface of degree $2$ of type $49$ does not exist for $\F_2$, $\F_3$, $\F_4$, $\F_5$, $\F_7$, $\F_8$, and exists for the other finite fields.

\item A del Pezzo surface of degree $2$ of type $31$ does not exist for $\F_2$, $\F_3$, $\F_4$, and exists for the other finite fields.

\item Del Pezzo surfaces of degree $2$ of types $40$, $50$, $53$, $55$, $60$ do not exist for $\F_2$, and exist for the other finite fields.

\item Del Pezzo surfaces of degree $2$ of types $43$, $44$, $45$, $52$, $54$, $57$, $59$ exist for all finite fields.

\item A del Pezzo surface of degree $2$ of type $35$ does not exist for $\F_2$, and exists for any~$\F_q$ where $q \geqslant 4$.

\item Del Pezzo surfaces of degree $2$ of types $51$, $58$ exist for any $\F_q$ where $q$ is odd.

\item A del Pezzo surface of degree $2$ of type $56$ exists for any $\F_q$ where $q = 6k + 1$.

\end{enumerate}

\end{theorem}
 
\begin{remark}

The author does not know, how to construct a del Pezzo surface of degree~$2$ of type $35$ over $\F_3$, or show that such surface does not exist. Maybe it is better to use a computer in this case.
The existence of del Pezzo surfaces of degree $2$ of types $51$, $56$, $58$ is equivalent to the existence of minimal cubic surfaces of certain types (see Lemma~\ref{X3blowup}). The restrictions on $q$ come from the paper \cite{RT16}, where minimal cubic surfaces are considered. The complete answer in these cases is not known.
\end{remark}

The plan of this paper is as follows.

In Section $2$ we consider minimal del Pezzo surfaces of degree $2$ which admit structure of conic bundles. For these cases we apply \cite[Theorem 2.11]{Ry05} and get a minimal conic bundle with singular fibres over points of required degrees. Then, if it is possible, we construct some birational links from these bundles to minimal del Pezzo surfaces admitting structure of conic bundle.

In Section $3$ we consider minimal del Pezzo surfaces of degree $2$ such that the Picard number $\rho(X)$ is equal to $1$. We define \textit{Geiser twist} (see Definition \ref{GeiserTwistDef}) which gives us correspondence between these surfaces and non-minimal del Pezzo surfaces of degree $2$ of certain types. Then we realise the obtained surfaces as the blowups of del Pezzo surfaces of higher degree at several points.

In Appendix A there is a table which gives the classification of cyclic subgroups of the Weyl group $W(E_7)$ and some properties of these subgroups.

The author is a Young Russian Mathematics award winner and would like to thank its sponsors and jury. Also the author is grateful to C.\,Shramov for many useful discussions which form the basis of this paper, to A.\,Duncan, S.\,Gorchinskiy and S.\,Rybakov for discussions about the theme of this work, and to B.\,Banwait, F.\,Fit\'{e}, D.\,Loughran for introducing their results which are very useful and allow the author to avoid many technical problems.

\section{The conic bundle case}

In this section we construct minimal del Pezzo surfaces of degree $2$ admitting structure of conic bundles. We use the following theorem.

\begin{theorem}[cf. {\cite[Theorem 2.11]{Ry05}}]
\label{CBexist}
Let $x_1$, $\ldots$, $x_s$ be a set of points on $B = \Pro^1_{\F_q}$ of possibly different degrees. Then there exists a relatively minimal conic bundle $X \rightarrow B$ with degenerate fibres over points $x_1$, $\ldots$, $x_s$ if and only if $s$ is even.
\end{theorem}

For del Pezzo surfaces of degree $2$ admitting a structure of a conic bundle there are exactly $6$ degenerate geometric fibres. So by Theorem \ref{CBexist} there are six possibilities:

\begin{enumerate}
\item[(31)] the degenerate fibres are over six $\F_q$-points;
\item[(35)] the degenerate fibres are over two $\F_q$-points and two points of degree $2$;
\item[(40)] the degenerate fibres are over three $\F_q$-points and a point of degree $3$;
\item[(43)] the degenerate fibres are over an $\F_q$-point and a point of degree $5$;
\item[(44)] the degenerate fibres are over a point of degree $2$ and a point of degree $4$;
\item[(45)] the degenerate fibres are over two points of degree $3$.
\end{enumerate}

The numeration of cases is taken from Table \ref{table1}.

\begin{remark}
\label{CBnotexist}
Case $(31)$ cannot be achieved for $\F_2$, $\F_3$ and $\F_4$ since there are no six \mbox{$\F_q$-points} on $\Pro^1_{\F_2}$, $\Pro^1_{\F_3}$ and $\Pro^1_{\F_4}$.
Case $(35)$ cannot be achieved for $\F_2$ since there are no two points of degree $2$ on~$\Pro^1_{\F_2}$.
\end{remark}

The main problem is that not any surface admitting a structure of a conic bundle with~$6$ degenerate fibres is a del Pezzo surface of degree $2$. The following proposition is well-known (see e.g. \cite[Chapter 8, Exercise 3]{Pr15}). We give a proof for convenience of the reader.

\begin{proposition}
\label{CBstructure}
Let $\pi: X \rightarrow B$ be a minimal conic bundle over $\Pro^1_{\F_q}$ with $6$ degenerate fibres. Then we have one of the following possibilities.
\begin{enumerate}
\item The surface $X$ is a del Pezzo surface of degree $2$ admitting two structures of conic bundles.
\item There is a geometrically irreducible $2$-section $D$ on $X$ such that $D^2 = -2$.
\item There are two geometrically irreducible sections $C_1$ and $C_2$ on $X$ such that \mbox{$C_1^2 = C_2^2 = -2$} and $C_1 \cdot C_2 = 1$.
\item There are four geometrically irreducible disjoint sections $C_1$, $C_2$, $C_3$ and $C_4$ on $X$ such that $C_1^2 = C_2^2 = C_3^3 = C_4^2 = -2$.
\item There are two geometrically irreducible disjoint sections $C_1$ and $C_2$ on $X$ such that $C_1^2 = C_2^2 = -3$.
\end{enumerate}
\end{proposition}

\begin{proof}

For a minimal conic bundle $X \rightarrow B$ the group $\Pic(X)$ is generated by $-K_X$ and~$F$, where $F$ is the class of fibre $X \rightarrow B$. In this basis one has $K_X^2 = 2$, $K_X \cdot F = -2$ and $F^2 = 0$.

Assume that the anticanonical linear system $|-K_X|$ is not nef. Then there exists a $\ka$-irreducible reduced curve $C$ such that $-K_X \cdot C < 0$. Thus the curve $C$ has a class $-aK_X - bF$ and  $a < b$, since $-K_X \cdot C < 0$. By Riemann--Roch theorem $\operatorname{dim} |-K_X| = 2$. Therefore $|-K_X| = |C + M|$, where $|M|$ is a moveable linear system of dimension $2$. One has $M \sim (a-1)K_X + bF$. Hence
$$
M^2 = 2(a-1)^2 -2(a-1)b = 2(a - 1)(a - 1 - b).
$$
\noindent This number can be non-negative only if $a = 1$, and the linear system $|M| \sim |bF|$ has dimension $2$ only if $b = 2$. Therefore $C \sim -K_X - 2F$. For the ariphmetic genus of $C$ one has
$$
2\mathrm{p}_a(C) - 2 = C \cdot (C + K_X) = -4.
$$
\noindent Therefore $C$ is geometrically reducible and consists of two disjoint geometrically irreducible sections with the selfintersection number $-3$. This is case $(5)$ of Proposition~\ref{CBstructure}.

Now assume that the anticanonical linear system $|-K_X|$ is nef but not ample. Then there exists a $\ka$-irreducible reduced curve $C$ such that $-K_X \cdot C = 0$. The curve $C$ has class $-aK_X - bF$ and consists of geometrically irreducible curves with the selfintersection number $-2$. One has $a = b$ since $-K_X \cdot C = 0$. The number of geometrically irreducible components of $C$ is no greater than $2a = C \cdot F$. Therefore one has $-2a^2=C^2 \geqslant -4a$, and $a \leqslant 2$.

If $a = 2$ then $C^2 = -8$, and $C$ consists of four disjoint geometrically irreducible sections with the selfintersection number $-2$. This is case $(4)$ of Proposition \ref{CBstructure}.

If $a = 1$ then $C^2 = -2$. If $C$ is geometrically reducible then it consists of two disjoint geometrically irreducible sections $C_1$ and $C_2$ such that $C_1^2 = C_2^2 = -2$ and $C_1 \cdot C_2 = 1$. This is case $(3)$ of Proposition \ref{CBstructure}.

If $C$ is geometrically irreducible then its selfintersection number is $-2$. This is case $(2)$ of Proposition \ref{CBstructure}.

If $|-K_X|$ is ample then $X$ is a del Pezzo surface of degree $2$ and the linear systems $|F|$ and $|-2K_X - F|$ give two conic bundle structures. This is case $(1)$ of Proposition \ref{CBstructure}.

\end{proof}

To construct minimal del Pezzo surfaces of degree $2$ admitting a conic bundle structure we apply Theorem \ref{CBexist} and then construct a sequence of Sarkisov links ending at a del Pezzo surface of degree $2$. But it is not possible to construct such links in an arbitrary situation.

\begin{example}
\label{dP4last}
A minimal del Pezzo surface of degree $4$ over $\F_3$ admitting a structure of conic bundle with four degenerate fibres over $\F_3$-points does not exist, since such a surface should contain eight $\F_3$-points of intersection of $(-1)$-curves (three or more \mbox{$(-1)$-curves} can not meet each other at one point on a del Pezzo surface of degree $4$). But there are only four $\F_3$-points on a minimal conic bundle over $\Pro^1_{\F_3}$ with four degenerate fibres over $\F_3$-points. Nevertheless by Theorem \ref{CBexist} there exists a conic bundle with four smooth fibres over four $\F_3$-points on $\Pro^1_{\F_3}$. 
\end{example}

Example \ref{dP4last} improves results of \cite[Theorem 3.2]{Ry05}. The complete result about minimal del Pezzo surfaces of degree $4$ is the following.

\begin{theorem}[cf. {\cite[Theorem 3.2]{Ry05}}]
\label{RybakImprove}

In the notation of \cite{Ry05} the following holds.

\begin{enumerate}

\item Del Pezzo surfaces of degree $4$ with the zeta-functions $Z_1$, $Z_2$, $Z_{10}$, $Z_{18}$ exist for all finite fields.

\item A del Pezzo surface of degree $4$ with the zeta-function $Z_5$ does not exist for $\F_2$, and exists for the other finite fields.

\item A del Pezzo surface of degree $4$ with the zeta-function $Z_4$ does not exist for $\F_2$ and~$\F_3$, and exists for the other finite fields.

\end{enumerate}

\end{theorem}

\begin{proof}

Let us remind (see the proof of \cite[Theorem 3.2]{Ry05}) that the zeta-functions $Z_2$, $Z_4$ and $Z_5$ come from del Pezzo surfaces of degree $4$ which admit a conic bundle structure:

\begin{itemize}

\item the case $Z_2$ corresponds to a conic bundle with singular fibres over an $\F_q$-point and a point of degree $3$;

\item the case $Z_4$ corresponds to a conic bundle with singular fibres over four $\F_q$-points;

\item the case $Z_5$ corresponds to a conic bundle with singular fibres over two points of degree $2$.

\end{itemize}

In \cite[Theorem 3.2]{Ry05} it is proved that del Pezzo surfaces of degree $4$ with zeta-functions $Z_1$, $Z_{10}$ and $Z_{18}$ exist for any $\F_q$, and del Pezzo surfaces of degree $4$ with zeta-functions $Z_2$, $Z_4$ and $Z_5$ exist for any $\F_q$ where $q > 3$.

Del Pezzo surfaces of degree $4$ with the zeta-functions $Z_4$ and $Z_5$ do not exist over $\F_2$ since there are no four $\F_2$-points and two points of degree $2$ on $\Pro^1_{\F_2}$. From the proof of \cite[Theorem 3.2]{Ry05} one can see that for the other possibilities of $q$ there exists a del Pezzo surface of degree $4$ with the zeta-function $Z_2$, $Z_4$ or $Z_5$ if there exists a smooth fibre over an $\F_q$-point. Therefore a del Pezzo surface of degree $4$ with the zeta-function~$Z_2$ exists for any $\F_q$, and a del Pezzo surface of degree~$4$ with the zeta-function $Z_5$ exists for any $\F_q$, where $q \geqslant 3$. Example~\ref{dP4last} shows that a del Pezzo surface of degree $4$ with the zeta-function $Z_4$ does not exist over~$\F_3$.

\end{proof}

We want to know some facts about curves with negative selfintersection on conic bundles. 

\begin{proposition}
\label{CBblowup}
Assume that $X \rightarrow \Pro^1_{\ka}$ is a minimal conic bundle over arbitrary field $\ka$ with $n > 0$ degenerate geometric fibres. Then $\XX$ is isomorphic to a blowup of $\Pro^1_{\kka} \times \Pro^1_{\kka}$ at set of points $p_1$, ... $p_n$.
\end{proposition}

\begin{proof}
The conic bundle $X \rightarrow \Pro^1_{\ka}$ is minimal therefore there are at least two sections $C_1$ and $C_2$ with negative selfintersection $-k$ on $\XX$, since otherwise there is a unique section with negative selfintersection number and one can contract over $\ka$ all components of singular fibres meeting this section. If $k > n$ then we can contract $n$~components of singular fibres on $\XX$ and get a conic bundle $\overline{Y} \rightarrow \Pro^1_{\kka}$ without singular fibres. But the images of $C_1$ and $C_2$ on $\overline{Y}$ are curves with negative selfintersection. It is impossible since any conic bundle without singular fibres is either $\Pro^1_{\ka} \times \Pro^1_{\ka}$ or a Hirzebruch surface $\F_m$ and there is at most one curve with negative selfintersection.

If $k \leqslant n$ then we can blow down any $k$ components of singular fibres meeting with $C_1$ and for the other $n - k$ singular fibres blow down components not meeting $C_1$. Then we get a conic bundle $\overline{Y} \rightarrow \Pro^1_{\kka}$ without singular fibres, and the image of $C_1$ on this bundle is a curve with selfintersection $0$. But there are no curves with selfintersection $0$ on Hirzebruch surfaces $\F_m$ for $m > 0$. Thus $\overline{Y}$ is $\Pro^1_{\kka} \times \Pro^1_{\kka}$.
\end{proof}

\begin{corollary}
\label{CBPicard}
The Picard group $\Pic(\XX)$ of a minimal conic bundle $X \rightarrow \Pro^1_{\ka}$ is generated by the class of fibre $F$, a class of section $C$ such that $C^2 = 0$, and the classes of exceptional divisors $E_1$, $\ldots$, $E_n$. One has
$$
C \cdot F = 1, \qquad C \cdot E_i = 0, \qquad F \cdot E_i = 0.
$$
\end{corollary}

Note that any group acting on $\Pic(\XX)$ and preserving the conic bundle structure should preserve $F$ and $K_X = -2C - 2F + \sum \limits_{i=1}^n E_i$. Thus this group acts on the subspace $K_X^{\perp} \cap F^{\perp}$ of $\Pic(\XX) \otimes \mathbb{Q}$, that is the subspace of classes $H \in \Pic(\XX) \otimes \mathbb{Q}$ such that $H \cdot K_X = H \cdot F = 0$. This subspace is generated by
$$
F - E_1 - E_2, \quad E_1 - E_2, \quad \ldots, \quad E_{n-1} - E_n.
$$
Those generators form a root system of type $D_n$. The Weyl group $W(D_n)$ is isomorphic to $\left(\Z / 2\Z\right)^{n-1} \rtimes S_n$, where $S_n$ is a symmetric group of degree $n$. A subgroup $S_n$ permutes~$E_i$ and the normal group $\left(\Z / 2\Z\right)^{n-1}$ is generated by involutions $\iota_{ij}$ such that
$$
\iota_{ij}(C) = C + F - E_i - E_j, \quad \iota_{ij}(F) = F, \quad \iota_{ij} (E_i) = F - E_i, \quad \iota_{ij} (E_j) = F - E_j, \quad \iota_{ij} (E_k) = E_k
$$
\noindent for $k \ne i$, $k \ne j$.

Any element of $\left(\Z / 2\Z\right)^{n-1} \rtimes S_n$ has form $\iota_{i_1 \ldots i_{2k}} \cdot \sigma$, where $\sigma \in S_n$ is a permutation permutting $E_i$, and $\iota_{i_1 \ldots i_{2k}} \in \left(\Z / 2\Z\right)^{n-1}$ switches components of singular fibres over even number of points $p_{i_1}$, $\ldots$, $p_{i_{2k}}$ on the base $\Pro^1_{\kka}$.

In this notation for cases $(31)$, $(35)$, $(40)$, $(43)$, $(44)$, $(45)$ of minimal conic bundles with $6$~degenerate geometric fibres the group~$\Gamma$ is generated by an element conjugate to $\iota_{123456}$, $\iota_{1235}(34)(56)$, $\iota_{123456}(456)$, $\iota_{123456}(23456)$, $\iota_{13}(12)(3456)$ or $\iota_{123456}(123)(456)$ respectively.

Now we start constructing Sarkisov links of minimal conic bundles.

\begin{lemma}
\label{CBlink5}
Let $\pi: X \rightarrow B \cong \Pro^1_{\F_q}$ be a minimal conic bundle of type $(5)$ of Proposition~\ref{CBstructure}. There exists a birational map $f: X \dashrightarrow Y$ such that $\pi f^{-1}: Y \rightarrow B$ is a conic bundle that does not have sections with selfintersection number $-3$.
\end{lemma}

\begin{proof}
Applying Proposition \ref{CBblowup} we may assume that the two sections $C_1$ and $C_2$ with selfintersection~$-3$ have classes $C - E_1 - E_2 - E_3$ and $C - E_4 - E_5 - E_6$ respectively.

Let us show that there are no other sections with negative selfintersection number. Any section $D$ has class $C + aF - \sum\limits_{i=1}^6 b_iE_i$, where $a \geqslant 0$ and each $b_i$ is $0$ or $1$. One has
$$
D^2 = 2a - \sum\limits_{i=1}^6 b_i, \qquad C_1 \cdot D = a - b_1 - b_2 - b_3, \qquad C_2 \cdot D = a - b_4 - b_5 - b_6,
$$
\noindent therefore $D^2 = C_1 \cdot D + C_2 \cdot D \geqslant 0$.

The curves $C_1$ and $C_2$ are not defined over $\F_q$. Therefore if there exists an $\F_q$-point $P \in B$ such that $\pi^{-1}(P)$ is a smooth fibre, then any $\F_q$-point on the fibre~$\pi^{-1}(P)$ does not lie on a section with negative selfintersection number. So we can blow up such a point, contract the transform of $\pi^{-1}(P)$ and get a minimal conic bundle $Y \rightarrow B$ of type~$(3)$.

Such $\F_q$-point $P$ exists for all cases except $\Gamma = \langle \iota_{123456}(456) \rangle$ and $q = 2$, or $\Gamma = \langle \iota_{123456} \rangle$ and $q = 5$. In these cases we can find on $B$ a point $Q$ of degree $2$, and choose a point of degree $2$ on $\pi^{-1}(Q)$ that does not lie on a section with negative selfintersection number. We can blow up such a point, contract the transform of $\pi^{-1}(Q)$ and get a minimal conic bundle $Y \rightarrow B$ without sections with selfintersection less than $-2$.

\end{proof}

\begin{lemma}
\label{CBlink4}
Let $\pi: X \rightarrow B \cong \Pro^1_{\F_q}$ be a minimal conic bundle of type $(4)$ of Proposition~\ref{CBstructure}. There exists a birational map $f: X \dashrightarrow Y$ such that $\pi f^{-1}: Y \rightarrow B$ is a conic bundle of type $(1)$ or $(2)$ of Proposition \ref{CBstructure}.
\end{lemma}

\begin{proof}
Applying Proposition \ref{CBblowup} we may assume that the four sections $C_1$, $C_2$, $C_3$ and $C_4$ with selfintersection~$-2$ have classes $C - E_1 - E_2$, $C - E_3 - E_4$, $C - E_5 - E_6$ and $C + 2F - \sum\limits_{i=1}^6 E_i$ respectively.

Let us find other sections with negative selfintersection number. Any section $D$ has class $C + aF - \sum\limits_{i=1}^6 b_iE_i$, where $a \geqslant 0$ and each $b_i$ is $0$ or $1$. One has $D^2 = 2a - \sum\limits_{i=1}^6 b_i$, therefore $a \leqslant 2$. One can check that there are $8$ sections with selfintersection $-1$ on $X$. Their classes are $C + F -E_i - E_j - E_k$ where $i \in \{1, 2\}$, $j \in \{3, 4\}$, $k \in \{5, 6\}$.

Note that the element $\iota_{123456}$ maps $C_1$ to $C + F - E_3 - E_4 - E_5 - E_6$ that is not an effective divisor. Therefore $\Gamma$ is conjugate to $\iota_{1235}(34)(56)$ or $\iota_{13}(12)(3456)$. In these cases there are at most two orbits of the sections with selfintersection $-1$, and each of these orbits contains at most one $\F_q$-point.

The curves $C_1$, $C_2$, $C_3$ and $C_4$ are not defined over $\F_q$. Therefore on a smooth fibre over an $\F_q$-point $P$ there is an $\F_q$-point which does not lie on a section with negative selfintersection. So we can blow up such a point, contract the transform of $\pi^{-1}(P)$ and get a minimal conic bundle $Y \rightarrow B$ without sections with selfintersection less than $-1$. Such a conic bundle has type $(1)$ or~$(2)$.

\end{proof}

To construct links of minimal conic bundles for the cases $(3)$ and $(2)$ we need the following lemma.

\begin{lemma}
\label{BadPoints}
Let $\pi: X \rightarrow B \cong \Pro^1_{\F_q}$ be a minimal conic bundle of type $(3)$ or $(2)$ of Proposition~\ref{CBstructure}. The singular points of anticanonical curves lie on a divisor $C_1 + C_2 + R$ or $D + R$ respectively, where $R$ has class $2C + 3F - \sum\limits_{i=1}^6 E_i$. In particular, there are at most $4$ singular points of anticanonical curves on a fibre of $\pi$.
\end{lemma}

\begin{proof}
Note that both divisors $D$ in the case $(2)$ and $C_1 + C_2$ in the case $(3)$ have the class $2C + F - \sum\limits_{i=1}^6 E_i$. In both cases we denote these divisors by $W$.

Note that each point of $W$ is a singular point of an anticanonical curve of form $W + F$.

The surface $X$ is a weak del Pezzo surface, and the anticanonical linear system defines a separable map $f: X \rightarrow \Pro^2_{\F_q}$ of degree $2$ that contract $W$ to a point $P$. Anticanonical curves map to lines on $\Pro^2_{\F_q}$, and singular points on anticanonical curves come from points of intersection of these lines and the branch divisor of $f$. Therefore any singular point of an anticanonical curve lie on the ramification divisor of $f$. This divisor has class $-2K_X$ and consists of $W$ and $R \sim 2C + 3F - \sum\limits_{i=1}^6 E_i$.
\end{proof}

\begin{lemma}
\label{CBlink3}
Let $\pi: X \rightarrow B \cong \Pro^1_{\F_q}$ be a minimal conic bundle of type $(3)$ of Proposition~\ref{CBstructure}. Then there exists a birational map $f: X \dashrightarrow Y$ such that $\pi f^{-1}: Y \rightarrow B$ is a conic bundle of type $(1)$ or $(2)$ of Proposition \ref{CBstructure}.
\end{lemma}

\begin{proof}
Applying Proposition \ref{CBblowup} we may assume that the two sections $C_1$ and $C_2$ with selfintersection~$-2$ have classes $C - E_1 - E_2$ and $C + F - E_3 - E_4 - E_5 - E_6$ respectively.

Let us find other sections with negative selfintersection number. Any section $D$ has class $C + aF - \sum\limits_{i=1}^6 b_iE_i$, where $a \geqslant 0$ and each $b_i$ is $0$ or $1$. One has $D^2 = 2a - \sum\limits_{i=1}^6 b_i$, therefore $a \leqslant 2$. One can check that there are $20$ sections with selfintersection $-1$ on $X$. Their classes are $C - E_i$ where $i \in \{3, 4, 5, 6\}$, $C + F -E_i - E_j - E_k$ where $i \in \{1, 2\}$, $j \in \{3, 4, 5, 6\}$, $k \in \{3, 4, 5, 6\}$, and $C + 2F - \sum\limits_{i=1}^6 E_i + E_j$ where $j \in \{3, 4, 5, 6\}$.

If we find an $\F_q$-point $P$ on a smooth fibre which does not lie on any section with negative selfintersection, then we can blowup $X$ at $P$, blow down the strict transform of fibre containing $P$ and get a minimal conic bundle $Y \rightarrow B$ without sections with selfintersection less than $-1$. Such a conic bundle has type $(1)$ or $(2)$.

Let us find such a point for each possibility of $\Gamma$. %The group $\Gamma$ acts on the set of sections with selfintersection $-1$. Any $\F_q$-point on an $\Gamma$-orbit of such curve is an intersection point of all curves in this orbit.

If $\ord \Gamma = 4$ or $\ord \Gamma = 8$ then $\Gamma$ contains an element conjugate to $\iota_{1234}$. But such an element can not map $C_1$ to $C_2$. Therefore these cases are impossible.

In the other cases $\Gamma$ contains the element $\iota_{123456}$ that maps any section $T$ with selfintersection $-1$ to $-K_X - T$. Thus any $\F_q$-point on any section with selfintersection $-1$ is a singular point of an anticanonical curve and lie on $R$ (see Lemma \ref{BadPoints}).

If $q > 2$ then on the smooth fibre containing the point of intersection of $C_1$ and $C_2$ there are at least $3$ other $\F_q$-points. At most $2$ of these points lie on $R$. Therefore there is an $\F_q$-point $P$ on this fibre which does not lie on any section with negative selfintersection, and we are done.

If $q = 2$ then $\Gamma = \langle \iota_{1456} (23456) \rangle$ or $\Gamma = \langle \iota_{2456} (123)(456) \rangle$ since in the other two remaining cases there are at least four fibres over $\F_q$-points that is impossible. Therefore there is a fibre over an $\F_q$-point that does not contain the point of intersection of $C_1$ and~$C_2$. The divisor $C_1 + C_2$ intersects this fibre at a point of degree $2$. Therefore by Lemma \ref{BadPoints} there is an $\F_q$-point $P$ on this fibre which does not lie on $R$, and we are done.

\end{proof}

\begin{proposition}
\label{CBlink2}
Let $\pi: X \rightarrow B \cong \Pro^1_{\F_q}$ be a minimal conic bundle of type $(2)$ of Proposition~\ref{CBstructure}. Then there exists a birational map $f: X \dashrightarrow Y$ such that $\pi f^{-1}: Y \rightarrow B$ is a conic bundle of type $(1)$ of Proposition \ref{CBstructure} in all possible cases except the following cases: the group $\Gamma$ is conjugate to $\langle \iota_{123456}(456) \rangle$ and $q = 2$; the group $\Gamma$ is conjugate to $\langle \iota_{1235} (34)(56) \rangle$ and $q$ is $3$ or $4$; the group $\Gamma$ is conjugate to $\langle \iota_{123456} (23456) \rangle$ and $q = 2$.
\end{proposition}

To prove this proposition we need several lemmas.

\begin{lemma}
\label{CBlink2GoodPoint}
Let $\pi: X \rightarrow B \cong \Pro^1_{\F_q}$ be a minimal conic bundle of type $(2)$ of Proposition~\ref{CBstructure}. If there is an $\F_q$-point $P$ on a smooth fibre that is not a singular point of an anticanonical curve and does not lie on any section with negative selfintersection. Then there exists a birational map $f: X \dashrightarrow Y$ such that $\pi f^{-1}: Y \rightarrow B$ is a conic bundle of type $(1)$ of Proposition \ref{CBstructure}.
\end{lemma}

\begin{proof}

Applying Proposition \ref{CBblowup} we may assume that the $2$-section $D$ with selfintersection~$-2$ has the class $2C + F - \sum\limits_{i=1}^6 E_i$.

One can check that there are $32$ sections with selfintersection $-1$ on $X$. Their classes are $C - E_i$, $C + F -E_i - E_j - E_k$  and $C + 2F - \sum\limits_{i=1}^6 E_i + E_j$.

Let us show that there are no other $2$-sections with negative selfintersection number. Any $2$-section $H$ has class $2C + aF - \sum\limits_{i=1}^6 b_iE_i$, where $a \geqslant 1$ and each $b_i$ is $0$, $1$ or $2$. We can assume that $b_1 \geqslant b_2 \geqslant \ldots \geqslant b_6$. Note that
$$
H \cdot \left( C + 2F - \sum\limits_{i=1}^5 E_i \right) = a + 4 - \sum\limits_{i=1}^5 b_i \geqslant 0,
$$
$$
H \cdot \left( C + F - E_1 - E_2 - E_3 \right) = a + 2 - b_1 - b_2 - b_3 \geqslant 0,
$$
$$
H \cdot \left( C - E_1 \right) = a - b_1 \geqslant 0.
$$

Therefore if $a = 5$ then $b_5 \ne 2$, if $a = 4$ then either $b_4 \ne 2$ or $b_5 = 0$, if $a = 3$ then~$b_3 \ne 2$. But for $a \geqslant 6$ and in these cases one has $H^2 = 4a - \sum\limits_{i=1}^6 b_i^2 \geqslant 0$. Therefore $a = 1$ or $a = 2$.

If $a = 1$ then $b_1 = 1$ and $H^2 = D \cdot H \geqslant 0$. If $a = 2$ then $b_2 = 1$ or $b_3 = 0$. One has $H^2 < 0$ only for $b_1 = 2$ and $b_2 = b_3 = \ldots = b_6 = 1$. But in this case $D \cdot H = -1$ that is impossible. Thus $D$ is the only $2$-section with negative selfintersection number.

Assume that an $\F_q$-point $P$ on a smooth fibre does not lie on any section with negative selfintersection and not a singular point of an anticanonical curve. If we blow up $X$ at $P$ and blow down the strict transform of fibre containing $P$ then we get a minimal conic bundle $Y \rightarrow B$ without sections with selfintersection less than $-1$.

Assume that there is a $2$-section with selfintersection less than $-1$ on $Y$. Let $H \subset X$ be the preimage of such a $2$-section. Then $H$ intersects each component of each degenerate fibre at a point since $Y$ has type $(2)$. Therefore $H$ has class $2C + aF - \sum\limits_{i=1}^6 E_i$. If $a = 1$ then $H \cdot D = -2$. It means that $H = D$ and $P$ lies on $D$, but each point of $D$ is a singular point of an anticanonical curve $D + F$, so this is impossible.

Note that the multiciplicity of $P$ on $H$ is no greater than $2$. Therefore if $a > 2$ then $H^2 > 2$ and selfintersection of the transform of $H$ on $Y$ is greater than $-2$. Thus $a = 2$, and $H$ is an anticanonical curve with singularity at $P$. But $P$ is not a singular point of anticanonical curve. So we have contradiction, and there are no $2$-sections with selfintersection less than $-1$ on $Y$.

\end{proof}

\begin{lemma}
\label{CBlink2i123456F5}
There is no a minimal conic bundle $\pi: X \rightarrow B \cong \Pro^1_{\F_q}$ of type $(2)$ with $\ord \Gamma = 2$ over~$\F_5$.
\end{lemma}

\begin{proof}

If $\ord \Gamma = 2$ and $q = 5$ then $D$ contains six $\F_q$-points. But there are only six $\F_q$-points on $X$ which lie on the six singular fibres. Therefore the map $D \rightarrow B$ is a double cover branched at six points. This is impossible, since $D$ is a smooth rational curve. Thus this case does not occur.

\end{proof}

\begin{lemma}
\label{CBlink2Common}
Let $\pi: X \rightarrow B \cong \Pro^1_{\F_q}$ be a minimal conic bundle of type $(2)$ of Proposition~\ref{CBstructure}. There exists an $\F_q$-point $P$ on a smooth fibre that is not a singular point of an anticanonical curve and does not lie on any section with negative selfintersection in all possible cases except the following cases: the group $\Gamma$ is conjugate to $\langle \iota_{123456}(456) \rangle$ and $q$ is $2$ or $3$; the group $\Gamma$ is conjugate to $\langle \iota_{1235} (34)(56) \rangle$ and $q$ is $3$ or $4$; the group $\Gamma$ is conjugate to $\langle \iota_{123456} (23456) \rangle$ and $q = 2$.
\end{lemma}

\begin{proof}

Let us find such a point for each possibility of $\Gamma$. The group $\Gamma$ acts on the set of sections with selfintersection $-1$. Any $\F_q$-point on an $\Gamma$-orbit of such curve is an intersection point of all curves in this orbit. Therefore an orbit of length $2$ contains at most two $\F_q$-points, an orbit of length $4$ contains at most one $\F_q$-point, an orbit of length~$6$ or more does not contain $\F_q$-points.

Note that the element $\iota_{123456}$ maps any section $T$ with selfintersection $-1$ to \mbox{$-K_X - T$}. Thus if $\iota_{123456} \in \Gamma$ then all $\F_q$-points on sections with negative selfintersection are singular points of anticanonical curves. If $\ord \Gamma = 8$ then the orbits of sections with selfintersection~$-1$ consist of $8$ curves and do not contain any $\F_q$-points. If $\ord \Gamma = 4$ then the orbits of sections with selfintersection $-1$ consist of $4$ meeting each other curves. So this is the only case when an $\F_q$-point on a section with negative selfintersection can be not a singular point of an anticanonical curve, and there are at most $8$ such points, since there are $32$ sections with selfintersection $-1$.

Assume that $\ord \Gamma \ne 4$, there is a smooth fibre over an $\F_q$-point, and $q \geqslant 4$. Then there are $5$ or more $\F_q$-points on this fibre. At most four of these points are singular points of anticanonical curves by Lemma \ref{BadPoints}. Thus there is an $\F_q$-point $P$ on this fibre that is not a singular point of an anticanonical curve and does not lie on any section with negative selfintersection, and we are done.

The latest assumption does not hold if $\ord \Gamma = 2$ and $q = 5$, $\ord \Gamma \geqslant 6$ and $q \leqslant 3$, or $\ord \Gamma = 4$. The case $\ord \Gamma = 2$ and $q = 5$ does not occur by Lemma \ref{CBlink2i123456F5}.

If $q = 3$ and $\Gamma$ is conjugate to $\langle \iota_{123456} (23456) \rangle$, $\langle \iota_{123456} (123)(456) \rangle$ or $\langle \iota_{13} (12)(3456) \rangle$, then there are at least $3$ smooth fibres over $\F_3$-points. On these fibres there are at least twelve $\F_3$-points, at most $6$ of them lie on $R$ (see Lemma~\ref{BadPoints}) and at most $4$ lie on $D$. Thus there is an $\F_3$-point on a smooth fibre that is not a singular point of an anticanonical curve and does not lie on any section with negative selfintersection, and we are done.

If $q = 2$ and $\Gamma$ is conjugate to $\langle \iota_{123456} (123)(456) \rangle$ or $\langle \iota_{13} (12)(3456) \rangle$, then there are $3$ smooth fibres over $\F_2$-points and nine $\F_2$-points on $X$. The anticanonical linear system $|-K_X|$ contains $7$ elements, and three of these elements have form $D + F$. The other four elements have at most five singular points, since on the set of negative sections there is one $\Gamma$-orbit of length $2$ for $\Gamma = \langle \iota_{123456} (123)(456) \rangle$ and no such orbits for \mbox{$\Gamma = \langle \iota_{13} (12)(3456) \rangle$}, and irreducible anticanonical curves have at most one singular point. The curve $D$ contains three $\F_2$-points. Thus there is an $\F_2$-point on a smooth fibre that is not a singular point of an anticanonical curve and does not lie on any section with negative selfintersection, and we are done.

%If $q = 2$ and $\Gamma$ is conjugate to $\langle \iota_{123456} (23456) \rangle$ then there are $2$ singular fibres over $\F_2$-points and one smooth fibre over an $\F_2$-point. The anticaninonical linear system $|-K_X|$ contains $7$ elements, which correspond to points on $\Pro^2_{\F_2}$. For any two elements in this system we can construct the third element corresponding to the third point on a line passing through a pair of points on $\Pro^2_{\F_2}$. Three elements of $|-K_X|$ corresponding to a line in $\Pro^2_{\F_2}$ have form $D + F$, and one element is reducible curve consisting of two sections with selfintersection $-1$ which classes are $C - E_1$ and $C + 2F - E_2 - E_3 - E_4 - E_5 - E_6$.

%Assume that each $\F_2$-point on each smooth fibre over $\F_2$-point either lie on $D$ or is a singular point of an anticanonical curve. 

Now assume that $\ord \Gamma = 4$. If $q \geqslant 5$ then there are $q - 1$ smooth fibres over $\F_q$-points. On these fibres there are at least $q^2 - 1$ points and at least $(q - 1)^2$ of them do not lie on~$R$ (see Lemma \ref{BadPoints}). The curve $D$ contains $q + 1$ points defined over $\F_q$ and at most eight $\F_q$-points lie on sections with negative selfintersection. One has $(q - 1)^2 > q + 9$ for $q \geqslant 5$. Thus there is an $\F_q$-point on a smooth fibre that is not a singular point of anticanonical curve and does not lie on any section with negative selfintersection, and we are done.

\end{proof}

\begin{lemma}
\label{CBlink2ord4F3}
Let $\pi: X \rightarrow B \cong \Pro^1_{\F_3}$ be a minimal conic bundle of type $(2)$ of Proposition~\ref{CBstructure} and $\Gamma$ is conjugate to $\langle \iota_{123456} (456) \rangle$ over $\F_3$. There exists an $\F_3$-point $P$ on a smooth fibre that is not a singular point of an anticanonical curve and does not lie on any section with negative selfintersection.
\end{lemma}

\begin{proof}

There are $3$ singular fibres over $\F_3$-points and one smooth fibre over an $\F_3$-point. The curve $D$ contains four $\F_3$-points, and at least two of these points lie on the singular fibres. If there is less than four singular anticanonical curves with singular $\F_3$-points on the smooth fibre then there is an $\F_3$-point on this fibre that is not a singular point of an anticanonical curve, and does not lie on any section with negative selfintersection, and we are done. Let us show that there can not be four singular $\F_3$-points of anticanonical curves on the smooth fibre.

Assume that there exists an irreducible singular anticanonical curve $A$ defined over~$\F_3$. Then $A$ contains at least three $\F_3$-points and at least two of these points lie on the singular fibres. Thus $A$ and $D$ have a common point and we have contradiction, since $-K_X \cdot D = 0$. Therefore only reducible anticanonical curves can have singular $\F_3$-points. If such a curve consisting of two sections with selfintersection $-1$ contains two $\F_3$-points then one of these points lies on a singular fibre over an $\F_3$-point. But $-K_X \cdot F = 2$, so it is impossible. Thus two $\F_3$-points on the smooth fibre not lying on $D$ can be singular points of an anticanonical curves only if these curves $A_1$ and $A_2$ are reducible curves, consisting of sections with selfintersection $-1$ which are tangent.

In this case let us consider the anticanonical map $\varphi: X \rightarrow \Pro^2_{\F_3}$. This map has degree~$2$ and the branch divisor is a singular plane quartic curve $B'$ with an ordinary double point $\varphi(D)$. The images $\varphi(A_1)$ and $\varphi(A_2)$ are lines each of which intersects $B'$ at a point with multiplicity $4$. One can choose coordinates on $\Pro^2_{\F_3}$ such that the points $A_1 \cap A_2$, $A_1 \cap B'$, $A_2 \cap B'$ and $\varphi(D)$ have coordinates $(0 : 0 : 1)$, $(1 : 0 : 0)$, $(0 : 1 : 0)$ and $(1 : 1 : 0)$ respectively, since the points $A_1 \cap B'$, $A_2 \cap B'$ and $\varphi(D)$ lie on a line which is the inage of the smooth fibre. In these coordinates $B'$ is given by the equation
$$
xy\left(U(x-y)^2 + V(x - y)z + Wz^2\right) - z^4 = 0.
$$
\noindent The preimages of lines $x = y$, $z = x - y$ and $z = y - x$ on $X$ are singular fibres. Therefore each of these lines is either bitangent to $B'$ or passes through $\varphi(D)$ with multiplicity greater than $2$. Thus we have $V = W = 0$, $U =-1$. But the singular point $( 1 : 1 : 0 )$ on the curve $xy(x-y)^2 + z^4 = 0$ is not a node. This contradiction finishes the proof.

\end{proof}

Now we can prove Proposition \ref{CBlink2}.

\begin{proof}[Proof of Proposition \ref{CBlink2}]

By Lemmas \ref{CBlink2GoodPoint} and \ref{CBlink2Common} the map $X \dashrightarrow Y$, where $Y$ is a conic bundle of type $(1)$ of Proposition~\ref{CBstructure}, exists for all possible cases except the following cases: the group $\Gamma$ is conjugate to $\langle \iota_{123456}(456) \rangle$ and $q$ is $2$ or $3$; the group $\Gamma$ is conjugate to $\langle \iota_{1235} (34)(56) \rangle$ and $q$ is $3$ or $4$; the group $\Gamma$ is conjugate to $\langle \iota_{123456} (23456) \rangle$ and $q = 2$.

By Lemmas \ref{CBlink2GoodPoint} and \ref{CBlink2ord4F3} such a map exists for $\Gamma$ conjugate to $\langle \iota_{123456} (456) \rangle$ over $\F_3$.

\end{proof}

Now we collect results of this section in the following proposition.

\begin{proposition}
\label{main2}
In the notation of Table \ref{table1} the following holds.

\begin{enumerate}
\item[(i)] A del Pezzo surface of degree $2$ of type $31$ does not exist for $\F_2$, $\F_3$, $\F_4$, and exists for the other finite fields.

\item[(ii)] A del Pezzo surface of degree $2$ of type $35$ does not exist for $\F_2$, and exists for any~$\F_q$ where $q \geqslant 4$.

\item[(iii)] Del Pezzo surfaces of degree $2$ of types $40$, $43$ exist for any $\F_q$ where $q \geqslant 3$.

\item[(iv)] Del Pezzo surfaces of degree $2$ of types $44$, $45$ exist for all finite fields.

\end{enumerate}
\end{proposition}

\begin{proof}

The surfaces of type $31$ do not exist for $\F_2$, $\F_3$, $\F_4$ and the surfaces of type $35$ do not exist for $\F_2$ by Remark \ref{CBnotexist}.

We apply Theorem \ref{CBexist} for each remaining case and then consequently apply Lemmas~\ref{CBlink5}, \ref{CBlink4}, \ref{CBlink3} and Proposition \ref{CBlink2}. Then we get a minimal del Pezzo surface $X$ of degree~$2$ in all cases except the following: the surface $X$ has type $40$ and $q = 2$; the surface~$X$ has type $35$ and $q$ is $3$ or $4$; the surface $X$ has type $43$ and $q = 2$.

The surface of type $35$ exists over $\F_4$ since the surface of type $44$ exists over $\F_2$, and the other cases are excluded by the conditions of this proposition.

\end{proof}

Del Pezzo surfaces of types $40$ and $43$ over $\F_2$ are considered in Section $3$.

\section{The case $\rho(\XX)^{\Gamma} = 1$}

In this section we construct minimal del Pezzo surfaces of degree $2$ with the Picard number $1$. In this case a del Pezzo surface $X$ is not a blow up of del Pezzo surface of higher degree and does not admit a structure of conic bundle.

For a del Pezzo surface $X$ of degree $2$ the linear system $|-K_X|$ gives a double cover of~$\Pro^2_{\F_q}$. This cover defines an involution $\gamma$ on $X$ which is called {\it Geiser involution}. Therefore we can apply the following proposition.

\begin{proposition}[{\cite[Proposition 4.4]{RT16}}]
\label{dPtwist}
Let $X_1$ be a smooth algebraic variety over finite field $\F_q$ such that a cyclic group $G$ of order $n$ acts on $X_1$ and this action induces a faithful action of $G$ on the group $\Pic(\XX_1)$. Let $\Gamma_1$ be the image of the Galois group $\Gal\left(\overline{\F}_q / \F_q \right)$ in the group $\Aut\left(\Pic(\XX_1)\right)$. Let $h$ and $g$ be the generators of $\Gamma_1$ and $G$ respectively.

Then there exists a variety $X_2$ such that the image $\Gamma_2$ of the Galois group $\Gal\left(\overline{\F}_q / \F_q \right)$ in the group $\Aut\left(\Pic(\XX_2)\right) \cong \Aut\left(\Pic(\XX_1)\right)$ is generated by the element~$gh$.
\end{proposition}

Note that in Proposition \ref{dPtwist} one has $\XX_1 \cong \XX_2$. Therefore if $X_1$ is a del Pezzo surface, then $X_2$ is a del Pezzo surface of the same degree.

\begin{definition}
\label{GeiserTwistDef}
Let $X_1$ be a del Pezzo surface of degree $2$, such that the image $\Gamma_1$ of the Galois group $\Gal\left(\overline{\F}_q / \F_q \right)$ in the group $\Aut\left(\Pic(\XX_1)\right)$ is generated by an element $h$. Then by Proposition \ref{dPtwist} there exists a del Pezzo surface $X_2$ of degree $2$, such that the image $\Gamma_2$ of the Galois group $\Gal\left(\overline{\F}_q / \F_q \right)$ in the group $\Aut\left(\Pic(\XX_2)\right)$ is generated by an element~$\gamma h$. We say that the surface $X_2$ is a {\it Geiser twist} of the surface $X_1$.
\end{definition}

Note that Geiser twists are also used in the paper \cite[see 4.1.2]{BFL16}.

\begin{remark}
\label{twisteigenvalues}
Note that the Geiser involution $\gamma$ acts on $K_X^{\perp}$ by multiplying all elements by $-1$. Therefore the eigenvalues of the group $\Gamma_2$ are the eigenvalues of the group $\Gamma_1$ multiplied by $-1$. Thus for each type of the group $\Gamma_1$ it is easy to find the type of the corresponding group $\Gamma_2$ (see Table \ref{table1}), except the cases where two types of $\Gamma_2$ have the same collections of eigenvalues. These cases do not appear in this paper.
\end{remark}

Now we consider the remaining cases of Section $2$.

\begin{lemma}
\label{CB1_5F2}
A del Pezzo surface of degree $2$ of type $43$ exists for any field $\F_q$.
\end{lemma}

\begin{proof}

By \cite[Section~3, case $a = 0$]{BFL16} for any $q$ there exists a del Pezzo surface $X_1$ of degree $2$ that is the blowup of a point of degree $2$ and a point of degree $5$ on $\Pro^2_{\F_q}$. The surface $X_1$ has type $24$ since the generator of the group $\Gamma_1$ has eigenvalues $1$, $-1$, $1$, $\xi_5$, $\xi_5^2$, $\xi_5^3$, $\xi_5^4$ on $K_X^{\perp} \subset \Pic(\XX) \otimes \mathbb{Q}$, where $\xi_5$ is a fifth root of unity. By Remark \ref{twisteigenvalues} for the Geiser twist $X_2$ of $X_1$ the generator of the group $\Gamma_2$ has eigenvalues $-1$, $1$, $-1$, $-\xi_5$, $-\xi_5^2$, $-\xi_5^3$, $-\xi_5^4$ and $X_2$ has type $43$ (see Table \ref{table1}).

\end{proof}

\begin{lemma}
\label{CB1_1_1_3F2}
A del Pezzo surface of degree $2$ of type $40$ does not exist for $\F_2$.
\end{lemma}

\begin{proof}

Assume that a del Pezzo surface $X_1$ of degree $2$ of type $40$ exists for $\F_2$. Then the generator of the group $\Gamma_1$ has eigenvalues $1$, $-1$, $-1$, $-1$, $-1$, $-\omega$, $-\omega^2$ on \mbox{$K_X^{\perp} \subset \Pic(\XX) \otimes \mathbb{Q}$}, where $\omega$ is a third root of unity. By Remark \ref{twisteigenvalues} for the Geiser twist $X_2$ of~$X_1$ the generator of the group $\Gamma_2$ has eigenvalues $-1$, $1$, $1$, $1$, $1$, $\omega$, $\omega^2$ and $X_2$ has type~$7$ (see Table \ref{table1}). Thus $X_2$ is the blowup of $\Pro^2_{\F_2}$ at two $\F_q$-points $p_1$ and $p_2$, a point~$p_3$ of degree $2$, and a point $p_4$ of degree $3$. The line $L$ passing through $p_3$ and the conic~$C$ passing through $p_3$ and $p_4$ are defined over $\F_2$ and do not have common $\F_2$-points. Therefore there are eight different $\F_2$-points on $\Pro^2_{\F_2}$: three $\F_2$-points on $C$, three $\F_2$-points on $L$, $p_1$ and $p_2$. That is impossible.

\end{proof}

Now let us consider Geiser twists of del Pezzo surfaces $X$ with $\rho(\XX)^{\Gamma} = 1$.

\begin{proposition}
\label{GeiserTwist}
Let $X$ be a del Pezzo surface of degree $2$ over $\F_q$ such that $\rho(\XX)^{\Gamma} = 1$. Then $X$ exists if and only if there exists a del Pezzo surface $X'$ of degree $2$ such that (we use the notation Table \ref{table1}) the following holds:

\begin{enumerate}

\item if $X$ has type $60$ then $X'$ has type $32$. Therefore $X'$ is a blowup of a cubic surface of type $(c_{11})$ (see \cite{SD67}) at an $\F_q$-point;

\item if $X$ has type $59$ then $X'$ has type $36$. Therefore $X'$ is a blowup of a minimal del Pezzo surface of degree $5$ at a point of degree $3$;

\item if $X$ has type $58$ then $X'$ has type $46$. Therefore $X'$ is a blowup of a cubic surface of type $(c_{13})$ (see \cite{SD67}) at an $\F_q$-point;

\item if $X$ has type $57$ then $X'$ has type $39$. Therefore $X'$ is a blowup of $\Pro^2_{\F_q}$ at a point of degree $7$;

\item if $X$ has type $56$ then $X'$ has type $47$. Therefore $X'$ is a blowup of a cubic surface of type $(c_{14})$ (see \cite{SD67}) at an $\F_q$-point;

\item if $X$ has type $55$ then $X'$ has type $12$. Therefore $X'$ is a blowup of $\Pro^2_{\F_q}$ at two points of degree $3$ and an $\F_q$-point;

\item if $X$ has type $54$ then $X'$ has type $15$. Therefore $X'$ is a blowup of $\Pro^2_{\F_q}$ at a point of degree $5$ and an $\F_q$-point;

\item if $X$ has type $53$ then $X'$ has type $4$. Therefore $X'$ is a blowup of $\Pro^2_{\F_q}$ at a point of degree $3$ and four $\F_q$-points;

\item if $X$ has type $52$ then $X'$ has type $44$. Therefore $X'$ is a minimal del Pezzo surface of degree $2$ admitting a conic bundle structure with degenerate fibres over points of degree $2$ and $4$;

\item if $X$ has type $51$ then $X'$ has type $48$. Therefore $X'$ is a blowup of a cubic surface of type $(c_{12})$ (see \cite{SD67}) at an $\F_q$-point;

\item if $X$ has type $50$ then $X'$ has type $17$. Therefore $X'$ is a blowup at two $\F_q$-points of a minimal del Pezzo surface of degree $4$ admitting a conic bundle structure with degenerate fibres over two points of degree $2$;

\item if $X$ has type $49$ then $X'$ has type $1$. Therefore $X'$ is a blowup of $\Pro^2_{\F_q}$ at seven $\F_q$-points.

\end{enumerate}

\end{proposition}

\begin{proof}

By Remark \ref{twisteigenvalues} for each type of the group $\Gamma$ it is easy to find the type of the corresponding group $\Gamma'$ (see Table \ref{table1}).

Note that each considered type of the group $\Gamma$ has unique collection of eigenvalues of action on \mbox{$K_X^{\perp} \subset \Pic(\XX) \otimes \mathbb{Q}$}. Moreover, a blowup of a del Pezzo surface at a point of degree $d$ adds $\xi_d$, $\xi_d^2$, $\ldots$, $\xi_d^{d-1}$, $1$ to the collection of eigenvalues of the group $\Gamma$. Therefore each minded nonminimal del Pezzo surface $X'$ can be realised as a blowup of a del Pezzo surface at number of points of certain degrees.

\end{proof}

A del Pezzo surface of degree $2$ of type $44$ was constructed in Proposition \ref{main2} (iv).

To construct the other types of del Pezzo surfaces of degree $2$, such that $\rho(X) = 1$, it is sufficient to blow up a number of points of certain degrees on del Pezzo surfaces of higher degree. We apply the following well-known theorem.

\begin{theorem}[cf. {\cite[Theorem 2.5]{Man74}}]
\label{GenPos}
Let $1 \leqslant d \leqslant 9$, and $p_1$, $\ldots$, $p_{9-d}$ be $9-d$ geometric points on the projective plane $\Pro^2_{\kka}$ such that
\begin{itemize}
\item no three lie on a line;
\item no six lie on a conic;
\item for $d = 1$ the points are not on a singular cubic curve with singularity at one of these points.
\end{itemize}
Then the blowup of $\Pro^2_{\kka}$ at $p_1$, $\ldots$, $p_{9-d}$ is a del Pezzo surface of degree $d$.

Moreover, any del Pezzo surface $\XX$ of degree $1 \leqslant d \leqslant 7$ over algebraically closed field $\kka$ is the blowup of such set of points.
\end{theorem}

\begin{definition}
If for $1 \leqslant d \leqslant 9$ geometric points $p_1$, $\ldots$, $p_{9-d}$ on $\Pro^2_{\F_q}$ satisfy the conditions of Theorem \ref{GenPos}, then we say that the points $p_1$, $\ldots$, $p_{9-d}$ are in a \textit{general position.}
\end{definition}

\begin{corollary}
\label{dPblowup}
Let $\XX$ be a del Pezzo surface of degree $3 \leqslant d \leqslant 7$ and $p$ be a geometric point which does not lie on $(-1)$-curves. Then the blowup of $\XX$ at $p$ is a del Pezzo surface of degree $d-1$.
\end{corollary}

\begin{proof}
By Theorem \ref{GenPos} the surface $\XX$ is the blowup $f: \XX \rightarrow \Pro^2_{\kka}$ of points $p_1$, $\ldots$, $p_{9-d}$ on $\Pro^2_{\kka}$. Moreover, no three points in the set $f(p)$, $p_1$, $\ldots$, $p_{9-d}$ lie on a line, and no six points in this set lie on a conic, since $p$ does not lie on $(-1)$-curves. Thus the blowup of the points $f(p)$, $p_1$, $\ldots$, $p_{9-d}$ is a del Pezzo surface of degree $d - 1$ by Theorem \ref{GenPos}.
\end{proof}

Now we construct del Pezzo surfaces of degree $2$ of types $1$, $4$, $12$, $15$, $17$, $32$, $36$, $39$, $46$, $47$, $48$ which Geiser twists are minimal surfaces with $\rho(X)^{\Gamma} = 1$. 

\begin{lemma}
\label{P2blowup}
~
\begin{itemize}
\item A del Pezzo surface of degree $2$ of type $1$ does not exist for $\F_2$, $\F_3$, $\F_4$, $\F_5$, $\F_7$, $\F_8$, and exists for the other finite fields.

\item Del Pezzo surfaces of degree $2$ of types $4$, $12$ do not exist for $\F_2$, and exist for the other finite fields.

\item Del Pezzo surfaces of degree $2$ of types $15$, $39$ exist for all finite fields.
\end{itemize}
\end{lemma}

\begin{proof}

Considered types of del Pezzo surfaces of degree $2$ are blowups of $\Pro^2_{\F_{q}}$ at sets of points of certain degrees in a general position. Types $1$, $4$ and $15$ were considered in \cite[Subsection 4.2]{BFL16} in cases $a = 8$, $a = 5$ and $a = 3$ respectively. 

~

A del Pezzo surface of degree $2$ of type $12$ is the blowup of $\Pro^2_{\F_q}$ at two points of degree~$3$ and an $\F_q$-point. Such configuration of points in a general position does not exist for~$\F_2$ since one cannot blow up seven $\F_8$-points on $\Pro^2_{\F_8}$ in a general position (see \cite[Subsection~4.2, case $a = 8$]{BFL16}).

For $q \geqslant 3$ one can construct a cubic surface $S$ which is the blowup of $\Pro^2_{\F_q}$ at two points of degree $3$ in a general position (see \cite[Proposition 6.2]{RT16}). The $27$ lines on $S$ form $9$ triples defined over $\F_q$, and only $3$ of these triples consist of meeting each other lines. Therefore at most three $\F_q$-points on $S$ lie on the lines. So one can find an $\F_q$-point on $S$ which does not lie on the lines, blow up this point, and get a del Pezzo surface of type~$12$ by Corollary \ref{dPblowup}.

~

A del Pezzo surface of degree $2$ of type $39$ is the blowup of $\Pro^2_{\F_q}$ at a point of degree~$7$. Let $p_1 = \left(a^3 : a : 1 \right)$, where $a \in \F_{q^7} \setminus \F_q$, and $p_2$, $\ldots$, $p_7$ be the conjugates of $p_1$. If six points from the set $p_1$, $\ldots$, $p_7$ lie on a conic, then all these points lie on a conic defined over $\F_q$. But it is impossible, since for the conic given by $Ax^2 + Bxy + Cy^2 + Dxz + Eyz + Fz^2 = 0$ the equality $Aa^6 + Ba^4 + Ca^2 + Da^3 + Ea + F = 0$ holds only if $A = B = C = D = E = F = 0$.

Note that three points $\left( x^3 : x : 1 \right)$, $\left( y^3 : y : 1 \right)$, $\left( z^3 : z : 1 \right)$ lie on a line if and only if $x = y$, $x = z$, $y = z$ or $x + y + z = 0$. Therefore if three points in the set $p_1$, $\ldots$, $p_7$ lie on a line then $a^{q^i} + a^{q^j} + a^{q^k} = 0$. One may assume that $i = 0$. Up to symmetries there are four possibilities:
\begin{itemize}
\item $j = 1$, $k = 2$;
\item $j = 1$, $k = 3$;
\item $j = 1$, $k = 4$;
\item $j = 2$, $k = 4$.
\end{itemize}

Note that $a + a^q + \ldots a^{q^6} \in \F_q$.

If $j = 1$, $k = 2$ then $a + a^q + a^{q^2} = a^{q^3} + a^{q^4} + a^{q^5} = 0$, and $a^{q^6} \in \F_q$ that is impossible.

If $j = 1$, $k = 4$ then $a + a^q + a^{q^4} = a^{q^2} + a^{q^3} + a^{q^6} = 0$, and $a^{q^5} \in \F_q$ that is impossible.

If $j = 2$, $k = 4$ then $a + a^{q^2} + a^{q^4} = a^{q} + a^{q^3} + a^{q^5} = 0$, and $a^{q^6} \in \F_q$ that is impossible.

If $j = 1$, $k = 3$ then $a + a^q + a^{q^3} = a^q + a^{q^2} + a^{q^4} = a^{q^5} + a^{q^6} + a^q = 0$, and $2a^q \in \F_q$ that is possible only for even $q$. But in this case $a + a^q + \ldots a^{q^6} = 0$. One can put $a' = a + 1$, and have $a' + a'^q + \ldots a'^{q^6} = 1$. Now $a'^{q^i} + a'^{q^j} + a'^{q^k} \ne 0$ for any $i$, $j$ and $k$.

Therefore for any $\F_q$ we can find a point of degree $7$ on $\Pro^2_{\F_q}$ in a general position, blow up this point, and get a del Pezzo surface of type $39$ by Theorem \ref{GenPos}.
\end{proof}

\begin{lemma}
\label{X5blowup}
A del Pezzo surface of degree $2$ of type $36$ exists for all finite fields.
\end{lemma}

\begin{proof}

A del Pezzo surface of degree $2$ of type $36$ is the blowup of a minimal del Pezzo surface of degree $5$ at a point of degree $3$. One can blow up a point of degree $5$ lying on a conic in $\Pro^2_{\F_q}$, contract the transform of this conic and get a minimal del Pezzo surface of degree $5$.

Let $P$ and $Q$ be two conics on $\Pro^2_{\F_q}$ defined over $\F_q$; the point $p_1$ be a geometric point on $P$ defined over $\F_{q^5}$; the points $p_2$, $\ldots$, $p_5$ be the conjugates of $p_1$; the point $q_1$ be a geometric point on $Q$ defined over $\F_{q^3}$ which does not lie on $P$; and $q_2$ and $q_3$ be the conjugates of $q_1$. Let $F$ be the Frobenius automorphism of $\Pro^2_{\F_q}$:
$$
F(x : y : z) = (x^q : y^q : z^q).
$$

Assume that points $p_i$, $p_j$ and $q_k$ lie on a line. Then the points $F^5p_i = p_i$, $F^5p_j = p_j$ and $F^5q_k$ lie on the same line. Therefore the points $q_1$, $q_2$ and $q_3$ lie on a line. But this is impossible since any line meets $Q$ at $2$ or $1$ point. The same arguments show that three points $p_i$, $q_j$ and $q_k$ can not lie on a line.

If a conic passes through six points $p_i$, $p_j$, $p_k$, $q_1$, $q_2$ and $q_3$ then this conic passes through the points $p_1$, $\ldots$, $p_5$ since either the set $\{ Fp_i, Fp_j, Fp_k\}$ or the set $\{F^2p_i, F^2p_j, F^2p_k\}$ has two common points with the set $\{p_i, p_j, p_k\}$. If a conic $C$ passes through four points from the set $\{p_1, \ldots, p_5\}$ and two points from the set $\{q_1, q_2, q_3\}$ then it passes through all points from these sets since it has $5$ common points with the conics $F^3C$ and $F^5C$. All these cases are impossible since the points $q_1$, $q_2$ and $q_3$ do not lie on $P$.

If an irreducible plane cubic curve $C$ passes through the eight points $p_1$, $\ldots$, $p_5$, $q_1$, $q_2$, $q_3$ and has a singularity at one of these points then it has at least three singular points since $C \cdot FC \geqslant 10$ and $C \cdot F^2C \geqslant 10$. That is impossible.

Thus the points $p_1$, $\ldots$, $p_5$, $q_1$, $q_2$, $q_3$ lie in a general position. The blowup of $\Pro^2_{\F_q}$ at these points is a del Pezzo surface of degree $1$ by Theorem \ref{GenPos}. One can contract the transform of $P$, and get a del Pezzo surface of type $36$.

\end{proof}

\begin{lemma}
\label{X4blowup}
A del Pezzo surface of degree $2$ of type $17$ does not exist for $\F_2$, and exists for the other finite fields.
\end{lemma}

\begin{proof}

A del Pezzo surface of degree $2$ of type $17$ is the blowup at two $\F_q$-points of a minimal del Pezzo surface $S$ of degree $4$ admitting a conic bundle structure with degenerate fibres over two points of degree $2$. Such del Pezzo surface does not exist over $\F_2$ and exists for the other finite fields by \cite[Theorem 3.2]{Ry05}. Assume that $q \geqslant 3$. The surface $S$ admits two structures of conic bundles and each of $16$ lines is a component of a singular fibre of one of these conic bundles. These lines form four $\Gal\left(\overline{\F}_q / \F_q \right)$-orbits, each consisting of $4$ curves. Therefore there are no $\F_q$-points on the $(-1)$-curves. But there are $(q + 1)^2$ points defined over $\F_q$ on the smooth fibres. Let $f:\widetilde{S} \rightarrow S$ be the blowup of $S$ at an $\F_q$-point $P$. By Corollary \ref{dPblowup} the surface $\widetilde{S}$ is a cubic surface. There are three lines on $\widetilde{S}$ defined over $\F_q$: the exceptional divisor $E = f^{-1}(P)$, and the proper transforms $C_1$ and $C_2$ of fibres of two conic bundles structures passing through $P$.

Let $F$ be the Frobenius automorphism. We show that all other $F$-orbits of lines consist of $4$ lines. Let $L$ be a line on $\widetilde{S}$ that differs from $E$, $C_1$ and $C_2$. If $L \cdot E = 0$ then $f(L)$ is a $(-1)$-curve and the orbit of this curve consists of $4$ curves. Assume that $E \cdot L = 1$. Then $C_1 \cdot L = C_2 \cdot L = 0$ since $E + C_1 + C_2 \sim -K_{\widetilde{S}}$. It means that $f(L)$ is a section of any conic bundle on $S$. For any singular fibre this section must meet one component $D_1$ of this fibre at a point, and for the other component $D_2$ of this fibre $f(L) \cdot D_2 = 0$. But we have $F^2 D_1 = D_2$, therefore $F^2 f(L) \cdot D_2 = f(L) \cdot D_1 = 1$. Thus $F^2f(L) \ne f(L)$ and the orbit of $L$ consists of $4$ lines.

Four lines on a cubic surface can not have a common point. Therefore all $\F_q$ points on lines on $\widetilde{S}$ are contained in $E$, $C_1$ and $C_2$. Thus there are $q^2$ points defined over $\F_q$ not lying on a lines on $\widetilde{S}$. One can blow up one of these points, and get a del Pezzo surface of type~$17$ by Corollary \ref{dPblowup}.

\end{proof}

\begin{lemma}
\label{X3blowup}
~
\begin{itemize}
\item A del Pezzo surface of degree $2$ of type $32$ does not exist for $\F_2$, and exists for the other finite fields.

\item Del Pezzo surfaces of degree $2$ of types $46$, $47$, $48$ exist only for those finite fields for which exist minimal cubic surfaces of types $c_{13}$, $c_{14}$ and $c_{12}$ respectively (see e.g. \cite[IV.9, Table 1]{Man74}).

\end{itemize}
\end{lemma}

\begin{proof}

Considered types of del Pezzo surfaces of degree $2$ are blowups of minimal cubic surfaces at an $\F_q$-point. Note that the $\Gal\left(\overline{\F}_q / \F_q \right)$-orbits of lines on a cubic surface with length greater than~$3$ do not contain $\F_q$-points and any orbit of length $3$ can contain at most one $\F_q$-point.

Del Pezzo surfaces of degree $2$ of types $46$, $47$, $48$ are blowups at an $\F_q$-point of minimal cubic surfaces of types $c_{13}$, $c_{14}$ and $c_{12}$ respectively. For these cubic surfaces all orbits of lines have length $3$ or greater, moreover, there is at most one orbit of length $3$ (see \cite[IV.9, Table 1]{Man74}). It means that there is at most one $\F_q$-point lying on the lines on such cubic surfaces. But there are $q^2 + 1$, $q^2 + q + 1$ and $q^2 + 2q + 1$ points defined over~$\F_q$ on such cubic surfaces respectively. Thus in each of those cases there is an $\F_q$-point not lying on the lines. One can blow up this point, and get a del Pezzo surface of type $46$, $47$ or $48$ respectively by Corollary \ref{dPblowup}.

~

A Del Pezzo surface of degree $2$ of type $32$ is the blowup at an $\F_q$-point of a minimal cubic surface $S$ of type $c_{11}$. This type of cubic surface was constructed in \cite{SD10} for any finite field. Each $\Gal\left(\overline{\F}_q / \F_q \right)$-orbit of lines consist of three lines (see \cite[IV.9, Table~1]{Man74}. Therefore there are at most nine $\F_q$-points lying on the lines, and all these points are Eckardt points.

There are $q^2 - 2q + 1$ points defined over $\F_q$ on $S$. Moreover, each nonsingular cubic surface with $q^2 - 2q + 1$ points defined over $\F_q$ has type $c_{11}$. If $q > 4$ then we can find an $\F_q$-point on $S$ which does not lie on the lines, blow up this point, and get a del Pezzo surface of type~$32$ by Corollary \ref{dPblowup}.

If $q = 2$ then $S$ contains a unique $\F_2$-point. By direct computation one can check that any cubic surface containing a unique $\F_2$-point is isomorphic to the surface given by the following equation:
\begin{equation}\label{F2cubic}
x^3 + y^3 + z^3 + x^2y + y^2z + z^2x + xyz + z^2t + zt^2 = 0.
\end{equation}
\noindent The $\F_2$-point $(0:0:0:1)$ is an Eckardt point. Therefore all $\F_2$-points on $S$ are contained in the lines, and a del Pezzo surface of type $32$ does not exist over $\F_2$.

The cubic given by equation \eqref{F2cubic} considered over $\F_4$ has type $c_{11}$ and contains \mbox{$\F_4$-point} $(0 : \omega : \omega : 1)$, where $\omega^3 = 1$. This point is not an Eckardt point. Therefore one can blow up this point, and get a del Pezzo surface of type~$32$ by Corollary \ref{dPblowup}.

For $q = 3$ the cubic surface given by the equation
$$
x^2y + xy^2 + x^2z + xyz + y^2z - xyt - xzt - yzt - z^2t - zt^2 + t^3 = 0
$$
\noindent containts exactly four $\F_3$-points: $(1 : 0 : 0 : 0)$, $(0 : 1 : 0 : 0)$, $(1 : -1 : 0 : 0)$, $(0 : 0 : 1 : 0)$. Thus this surface has type $c_{11}$. One can check that the point $(0 : 0 : 1 : 0)$ is not an Eckardt point. Therefore one can blow up this point, and get a del Pezzo surface of type~$32$ by Corollary \ref{dPblowup}.
\end{proof}

Now we collect results about del Pezzo surfaces of degree $2$ with $\rho(\XX)^{\Gamma} = 1$% in the following proposition.

\begin{proposition}
\label{main3}
In the notation of Table \ref{table1} the following holds.

\begin{enumerate}

\item[(i)] A del Pezzo surface of degree $2$ of type $49$ does not exist for $\F_2$, $\F_3$, $\F_4$, $\F_5$, $\F_7$, $\F_8$, and exists for the other finite fields.

\item[(ii)] Del Pezzo surfaces of degree $2$ of types $50$, $53$, $55$, $60$ do not exist for $\F_2$, and exist for the other finite fields.

\item[(iii)] Del Pezzo surfaces of degree $2$ of types $52$, $54$, $57$, $59$ exist for all finite fields.

\item[(iv)] Del Pezzo surfaces of degree $2$ of types $51$, $58$ exist for any $\F_q$ where $q$ is odd.

\item[(v)] A del Pezzo surface of degree $2$ of type $56$ exists for any $\F_q$ where $q = 6k + 1$.

\end{enumerate}
\end{proposition}

\begin{proof}

We apply Proposition \ref{GeiserTwist} and then apply Lemma \ref{P2blowup} for types $49$, $53$, $54$, $55$ and~$57$; Lemma \ref{X5blowup} for type $59$; Lemma \ref{X4blowup} for type $50$; Lemma \ref{X3blowup} for types $51$, $56$, $58$, $60$; and Proposition \ref{main2} (iv) for type $52$.

By \cite[Theorem 1.2]{RT16} cubic surfaces of types $c_{12}$ and $c_{13}$ exist for odd $q$, and cubic surfaces of type $c_{14}$ exist for $q = 6k + 1$. These types of cubic surfaces correspond to types $51$, $58$ and $56$ of del Pezzo surfaces of degree $2$ respectively by Proposition \ref{GeiserTwist} and Lemma \ref{X3blowup}.

\end{proof}

Now we prove Theorem \ref{main}.

\begin{proof}

Case $(1)$ is Proposition \ref{main3} (i); case $(2)$ is Proposition \ref{main2} (i); case $(3)$ follows from Proposition \ref{main2} (iii), Lemma \ref{CB1_1_1_3F2} and Proposition \ref{main3} (ii);  case $(4)$ follows from Proposition \ref{main2} (iii), Lemma \ref{CB1_5F2}, Proposition \ref{main2} (iv) and Proposition \ref{main3} (iii); case~$(5)$ is Proposition \ref{main2} (ii); case $(6)$ is Proposition \ref{main3} (iv); and case $(7)$ is Proposition~\ref{main3}~(v).

\end{proof}

\appendix

\section{Conjugacy classes of elements in $W(E_7)$}

In the following table we collect some facts about conjugacy classes of elements in the Weyl group~$W(E_7)$. This table is based on \cite[Table $10$]{Car72}. The first column is a number of a conjugacy class in order of their appearence. Throughout the paper this number is called a \textit{type} of del Pezzo surface. The second column is a Carter graph corresponding to the conjugacy class (see \cite{Car72}). The third column is the order of an element. The fourth column is the collection of eigenvalues of the action of an element on $K_X^{\perp} \subset \Pic(\XX) \otimes \mathbb{Q}$. The fifth column is the invariant Picard number $\rho(\XX)^{\Gamma}$. The last column is a number of the corresponding conjugacy class after the Geiser twist (see Definition \ref{GeiserTwistDef}). We denote by $\omega$ a third root of unity and by $\xi_d$ a $d$-th root of unity.

\begin{center}
\begin{longtable}{|c|c|c|c|c|c|}

\hline
Number & Graph & Order & Eigenvalues & $\rho(\XX)^{\Gamma}$ & Geiser \\
\hline
\hline \endhead
1. & $\varnothing$ & $1$ & $1$, $1$, $1$, $1$, $1$, $1$, $1$ & $8$ & 49. \\
\hline
2. & $A_1$ & $2$ & $1$, $1$, $1$, $1$, $1$, $1$, $-1$ & $7$ & 31. \\
\hline
3. & $A_1^2$ & $2$ & $1$, $1$, $1$, $1$, $1$, $-1$, $-1$ & $6$ & 18. \\
\hline
4. & $A_2$ & $3$ & $1$, $1$, $1$, $1$, $1$, $\omega$, $\omega^2$ & $6$ & 53. \\
\hline
5. & $A_1^3$ & $2$ & $1$, $1$, $1$, $1$, $-1$, $-1$, $-1$ & $5$ & 9. \\
\hline
6. & $A_1^3$ & $2$ & $1$, $1$, $1$, $1$, $-1$, $-1$, $-1$ & $5$ & 10. \\
\hline
7. & $A_2 \times A_1$ & $6$ & $1$, $1$, $1$, $1$, $-1$, $\omega$, $\omega^2$ & $5$ & 40. \\
\hline
8. & $A_3$ & $4$ & $1$, $1$, $1$, $1$, $i$, $-1$, $-i$ & $5$ & 33. \\
\hline
9. & $A_1^4$ & $2$ & $1$, $1$, $1$, $-1$, $-1$, $-1$, $-1$ & $4$ & 5. \\
\hline
10. & $A_1^4$ & $2$ & $1$, $1$, $1$, $-1$, $-1$, $-1$, $-1$ & $4$ & 6. \\
\hline
11. & $A_2 \times A_1^2$ & $6$ & $1$, $1$, $1$, $-1$, $-1$, $\omega$, $\omega^2$ & $4$ & 27. \\
\hline
12. & $A_2^2$ & $3$ & $1$, $1$, $1$, $\omega$, $\omega^2$, $\omega$, $\omega^2$ & $4$ & 55. \\
\hline
13. & $A_3 \times A_1$ & $4$ & $1$, $1$, $1$, $i$, $-1$, $-i$, $-1$ & $4$ & 21. \\
\hline
14. & $A_3 \times A_1$ & $4$ & $1$, $1$, $1$, $i$, $-1$, $-i$, $-1$ & $4$ & 22. \\
\hline
15. & $A_4$ & $5$ & $1$, $1$, $1$, $\xi_5$, $\xi_5^2$, $\xi_5^3$, $\xi_5^4$ & $4$ & 54. \\
\hline
16. & $D_4$ & $6$ & $1$, $1$, $1$, $-1$, $-\omega^2$, $-1$, $-\omega$ & $4$ & 19. \\
\hline
17. & $D_4(a_1)$ & $4$ & $1$, $1$, $1$, $i$, $-i$, $i$, $-i$ & $4$ & 50. \\
\hline
18. & $A_1^5$ & $2$ & $1$, $1$, $-1$, $-1$, $-1$, $-1$, $-1$ & $3$ & 3. \\
\hline
19. & $A_2 \times A_1^3$ & $6$ & $1$, $1$, $-1$, $-1$, $-1$, $\omega$, $\omega^2$ & $3$ & 16. \\
\hline
20. & $A_2^2 \times A_1$ & $6$ & $1$, $1$, $-1$, $\omega$, $\omega^2$, $\omega$, $\omega^2$ & $3$ & 45. \\
\hline
21. & $A_3 \times A_1^2$ & $4$ & $1$, $1$, $i$, $-1$, $-i$, $-1$, $-1$ & $3$ & 13. \\
\hline
22. & $A_3 \times A_1^2$ & $4$ & $1$, $1$, $i$, $-1$, $-i$, $-1$, $-1$ & $3$ & 14. \\
\hline
23. & $A_3 \times A_2$ & $12$ & $1$, $1$, $i$, $-1$, $-i$, $\omega$, $\omega^2$ & $3$ & 42. \\
\hline
24. & $A_4 \times A_1$ & $10$ & $1$, $1$, $\xi_5$, $\xi_5^2$, $\xi_5^3$, $\xi_5^4$, $-1$ & $3$ & 43. \\
\hline
25. & $A_5$ & $6$ & $1$, $1$, $-\omega^2$, $\omega$, $-1$, $\omega^2$, $-\omega$ & $3$ & 37. \\
\hline
26. & $A_5$ & $6$ & $1$, $1$, $-\omega^2$, $\omega$, $-1$, $\omega^2$, $-\omega$ & $3$ & 38. \\
\hline
27. & $D_4 \times A_1$ & $6$ & $1$, $1$, $-1$, $-\omega^2$, $-1$, $-\omega$, $-1$ & $3$ & 11. \\
\hline
28. & $D_4(a_1) \times A_1$ & $4$ & $1$, $1$, $i$, $-i$, $i$, $-i$, $-1$ & $3$ & 35. \\
\hline
29. & $D_5$ & $8$ & $1$, $1$, $-1$, $\xi_8$, $\xi_8^3$, $\xi_8^5$, $\xi_8^7$ & $3$ & 41. \\
\hline
30. & $D_5(a_1)$ & $12$ & $1$, $1$, $i$, $-i$, $-\omega^2$, $-1$, $-\omega$ & $3$ & 34. \\
\hline
31. & $A_1^6$ & $2$ & $1$, $-1$, $-1$, $-1$, $-1$, $-1$, $-1$ & $2$ & 2. \\
\hline
32. & $A_2^3$ & $3$ & $1$, $\omega$, $\omega^2$, $\omega$, $\omega^2$, $\omega$, $\omega^2$ & $2$ & 60. \\
\hline
33. & $A_3 \times A_1^3$ & $4$ & $1$, $i$, $-1$, $-i$, $-1$, $-1$, $-1$ & $2$ & 8. \\
\hline
34. & $A_3 \times A_2 \times A_1$ & $12$ & $1$, $i$, $-1$, $-i$, $\omega$, $\omega^2$, $-1$ & $2$ & 30. \\
\hline
35. & $A_3^2$ & $4$ & $1$, $i$, $-1$, $-i$, $i$, $-1$, $-i$ & $2$ & 28. \\
\hline
36. & $A_4 \times A_2$ & $15$ & $1$, $\xi_5$, $\xi_5^2$, $\xi_5^3$, $\xi_5^4$, $\omega$, $\omega^2$ & $2$ & 59. \\
\hline
37. & $A_5 \times A_1$ & $6$ & $1$, $-\omega^2$, $\omega$, $-1$, $\omega^2$, $-\omega$, $-1$ & $2$ & 25. \\
\hline
38. & $A_5 \times A_1$ & $6$ & $1$, $-\omega^2$, $\omega$, $-1$, $\omega^2$, $-\omega$, $-1$ & $2$ & 26. \\
\hline
39. & $A_6$ & $7$ & $1$, $\xi_7$, $\xi_7^2$, $\xi_7^3$, $\xi_7^4$, $\xi_7^5$, $\xi_7^6$ & $2$ & 57. \\
\hline
40. & $D_4 \times A_1^2$ & $6$ & $1$, $-1$, $-\omega^2$, $-1$, $-\omega$, $-1$, $-1$ & $2$ & 7. \\
\hline
41. & $D_5 \times A_1$ & $8$ & $1$, $-1$, $\xi_8$, $\xi_8^3$, $\xi_8^5$, $\xi_8^7$, $-1$ & $2$ & 29. \\
\hline
42. & $D_5(a_1) \times A_1$ & $12$ & $1$, $i$, $-i$, $-\omega^2$, $-1$, $-\omega$, $-1$ & $2$ & 23. \\
\hline
43. & $D_6$ & $10$ & $1$, $-1$, $-\xi_5^3$, $-\xi_5^4$, $-1$, $-\xi_5$, $-\xi_5^2$ & $2$ & 24. \\
\hline
44. & $D_6(a_1)$ & $8$ & $1$, $i$, $-i$, $\xi_8$, $\xi_8^3$, $\xi_8^5$, $\xi_8^7$ & $2$ & 52. \\
\hline
45. & $D_6(a_2)$ & $6$ & $1$, $-\omega^2$, $-1$, $-\omega$, $-\omega^2$, $-1$, $-\omega$ & $2$ & 20. \\
\hline
46. & $E_6$ & $12$ & $1$, $\omega$, $\omega^2$, $-i\omega$, $-i\omega^2$, $i\omega$, $i\omega^2$ & $2$ & 58. \\
\hline
47. & $E_6(a_1)$ & $9$ & $1$, $\xi_9$, $\xi_9^2$, $\xi_9^4$, $\xi_9^5$, $\xi_9^7$, $\xi_9^8$ & $2$ & 56. \\
\hline
48. & $E_6(a_2)$ & $6$ & $1$, $\omega$, $\omega^2$, $-\omega^2$, $-\omega$, $-\omega^2$, $-\omega$ & $2$ & 51. \\
\hline
49. & $A_1^7$ & $2$ & $-1$, $-1$, $-1$, $-1$, $-1$, $-1$, $-1$ & $1$ & 1. \\
\hline
50. & $A_3^2 \times A_1$ & $2$ & $-1$, $i$, $-1$, $-i$, $i$, $-1$, $-i$ & $1$ & 17. \\
\hline
51. & $A_5 \times A_2$ & $6$ & $\omega$, $\omega^2$, $-\omega^2$, $\omega$, $-1$, $\omega^2$, $-\omega$ & $1$ & 48. \\
\hline
52. & $A_7$ & $8$ & $\xi_8$, $i$, $\xi_8^3$, $-1$, $\xi_8^5$, $-i$, $\xi_8^7$ & $1$ & 44. \\
\hline
53. & $D_4 \times A_1^3$ & $6$ & $-1$, $-1$, $-1$, $-1$, $-\omega^2$, $-1$, $-\omega$ & $1$ & 4. \\
\hline
54. & $D_6 \times A_1$ & $10$ & $-1$, $-1$, $-1$, $-\xi_5^3$, $-\xi_5^4$, $-\xi_5$, $-\xi_5^2$ & $1$ & 15. \\
\hline
55. & $D_6(a_2) \times A_1$ & $6$ & $-1$, $-\omega^2$, $-1$, $-\omega$, $-\omega^2$, $-1$, $-\omega$ & $1$ & 12. \\
\hline
56. & $E_7$ & $18$ & $-1$, $-\xi_9^5$, $-\xi_9^7$, $-\xi_9^8$, $-\xi_9$, $-\xi_8^2$, $-\xi_9^4$ & $1$ & 47. \\
\hline
57. & $E_7(a_1)$ & $14$ & $-\xi_7^4$, $-\xi_7^5$, $-\xi_7^6$, $-1$, $-\xi_7$, $-\xi_7^2$, $-\xi_7^3$ & $1$ & 39. \\
\hline
58. & $E_7(a_2)$ & $12$ & $-\omega^2$, $-1$, $-\omega$, $-i\omega$, $-i\omega^2$, $i\omega$, $i\omega^2$ & $1$ & 46. \\
\hline
59. & $E_7(a_3)$ & $30$ & $-\omega^2$, $-1$, $-\omega$, $-\xi_5^3$, $-\xi_5^4$, $-\xi_5$, $-\xi_5^2$ & $1$ & 36. \\
\hline
60. & $E_7(a_4)$ & $6$ & $-\omega^2$, $-1$, $-\omega$, $-\omega^2$, $-\omega$, $-\omega^2$, $-\omega$ & $1$ & 32. \\
\hline
\caption[]{\label{table1} Conjugacy classes of elements in $W(E_7)$}
\end{longtable}
\end{center}

\bibliographystyle{alpha}

%\bibliographystyle{unsrt}
%\bibliography{my_ref}
\end{document}